\documentclass[11pt]{article}
\usepackage[utf8]{inputenc}
\usepackage{amsbsy,amsfonts,amsmath,amssymb,amsthm
,bbm,enumerate,euscript,graphicx,color}

\usepackage{hyperref} 
\hypersetup{
    colorlinks=true,
    linkcolor=blue,
    citecolor=blue,
    urlcolor = magenta,
}

\def\R{\mathbb{R}}
\def\N{\mathbb{N}}

\def\Z{\mathbb{Z}}

\def\P{\mathbb{P}}

\def\p{\mathbf{p}}

\def\D{\mathrm{d}\,}

\newtheorem{theo}{Theorem}
\newtheorem{coro}{Corollary}

\newtheorem{prop}{Proposition}

\theoremstyle{definition}
\newtheorem{defn}{Definition}
\newtheorem*{proo}{Proof}

\numberwithin{equation}{section}
\newtheorem*{rema}{Remark}

\title{\bf {Void probabilities for perturbed lattices under mixing}}
\author{Daniela Flimmel
\thanks{e-mail: \texttt{daniela.flimmel@matfyz.cuni.cz}, address: Sokolovská 83, 186 75 Prague, Czech Republic}
\and
\footnotesize{Department of Probability and Mathematical Statistics, Charles University}}

\date{\today}

\begin{document}
\maketitle
\begin{abstract}
    We study void probabilities for stationary perturbed lattices with possibly dependent perturbations. Our main result gives a general upper bound in terms of two ingredients: the number of lattice sites having a uniformly positive probability of hitting the observation set, and the $\alpha$-mixing rate of the perturbation field. For unions of lattice cubes, this yields explicit bounds without moment or regularity assumptions on the marginal distribution of the perturbations. As a consequence, we obtain quantitative estimates for growing bounded Borel sets with non-empty interior, including stretched-exponential decay under exponential mixing and polynomial decay under polynomial mixing.
\end{abstract}

\section{Introduction}
Perturbed lattices form a natural class of point processes lying between perfectly ordered lattices and more disordered spatial models. Starting from a deterministic lattice, each site is displaced by a random vector, possibly with dependence between the displacements. Despite the simplicity of the construction, perturbed lattices exhibit a rich range of spatial behaviour and have appeared in several contexts in probability, statistical physics and spatial statistics. For instance, they have been used as tractable models for particle systems and cosmological initial configurations (\cite{GJL02,GBJMSL06}). More recently, they have been considered as flexible models for regular spatial point patterns. In particular, Gaussian perturbation fields lead to computationally tractable models whose second-order characteristics can be evaluated explicitly and which can be fitted to spatial data \cite{F26}. Related extensions incorporating Gibbsian interactions have also recently been proposed in \cite{CR26}.

A particularly active motivation for studying perturbed lattices comes from hyperuniformity. A stationary point process is hyperuniform if its number variance grows at a slower rate than the volume of a large observation window. See \cite{Tor18, LR25} for broad overviews. Perturbed lattices constitute one of the basic classes of hyperuniform point processes and have been studied extensively from this perspective, see for instance \cite{G04, KT20, M26}. Most of the literature on perturbed lattices concerns independent
perturbations, while stationary dependent perturbation fields have received
considerably less attention. Recent results show that dependence may
substantially affect even basic large-scale properties such as
hyperuniformity \cite{DFHL24,F26}. In the present paper we investigate
another fundamental characteristic, namely void probabilities, under
general mixing assumptions on the perturbation field.

A further motivation comes from stabilization methods in stochastic geometry. Stabilization is a tool for proving laws of large numbers and central limit theorems for geometric statistics such as nearest-neighbour graphs and Voronoi-type functionals (see e.g. \cite{PY01,LRSY19}). Such arguments depend on controlling the tail of a random radius beyond which distant points no longer influence a local geometric score. Frequently, a large stabilization radius is connected with the presence of a large empty region. Quantitative void-probability estimates may therefore provide a useful ingredient for developing stabilization theory for perturbed lattices. We do not pursue such applications here, but this perspective provides an additional reason for seeking estimates that remain valid under dependence.

In the present paper we focus on \emph{void probabilities},
\[
\mathcal{P}(B):=\P(\Xi(B)=0),
\]
that is, the probability that an observation set $B$ contains no point of the process. Void probabilities are among the basic distributional characteristics of a point process. For balls, void probabilities are directly related to the spherical contact
distribution function (see \cite{CSKM13})
\[
H(r):=\P\bigl(d(x,\Xi)\leq r\bigr)
=1-\mathcal{P}(B(x,r)).
\]

Our main aim is to obtain quantitative void-probability estimates for
perturbed lattices whose displacement field is stationary but not
necessarily independent. In contrast with the independent setting, our
bounds allow general dependence controlled only through mixing
coefficients and, for suitable observation sets, require no moment
or regularity assumptions on the marginal perturbation distribution.

For unions of lattice cubes, the required good sites can be found without any moment, density or support assumption on the marginal perturbation distribution. This gives explicit void-probability bounds depending only on the mixing rate. By approximating growing observation sets from inside by lattice cubes, we obtain corresponding estimates for arbitrary bounded Borel sets with non-empty interior. In particular, exponential mixing leads to stretched-exponential decay of the void probability, while polynomial mixing yields polynomial decay up to a logarithmic factor. Finite-range dependence gives an exponential-in-volume bound.

We conclude by returning to independent perturbations. In this case the hitting events factorize conditionally on the stationarizing shift, and the void probability admits an exact product representation. This allows one to replace the separated-sites argument by the total hitting mass and yields bounds for arbitrary bounded Borel sets. Our estimates are less precise than the distribution-sensitive asymptotics obtained in \cite{ESY26}, but they require no regularity assumptions on the perturbation law and serve to highlight the qualitative difference between the independent and dependent settings. The novelty of this paper lies in the treatment of general class of perturbed lattices allowing also heavy tailed distribution of the perturbations and
dependencies within the perturbation field.

The paper is organized as follows. Section~2 introduces stationary perturbed lattices, void probabilities and the mixing coefficients used throughout. Section~3 contains the main estimates and their consequences under different mixing assumptions, together with the discussion of independent perturbations. The proofs are collected in the final section.

\section{Definitions}
We assume that all random variables are defined on a common probability space $(\Omega,\mathcal A,\mathbb P)$, which remains fixed throughout the paper. 
Moreover, for a Borel set $B$, we denote by $|B|$ its Lebesgue measure. In particular, by $B(x,r)$, we denote the closed Euclidean ball centered at $x \in \R^d$ with radius $r>0$. Finally, $\#(\cdot)$ denotes the cardinality and $\|\cdot\|$ is the Euclidean norm in $\R^d$.

\begin{defn}[Perturbed lattice]\label{D_PerturbedLattice}
    Let $\p=\{p_i, i \in \Z^d\}$ be a strictly stationary random field in $\R^d$, i.e.
    $$(p_i)_{i \in \Z^d} \stackrel{D}{=} (p_{i+j})_{i \in \Z^d}$$
    for all $j \in \Z^d$.
    Let $U$ be a uniform random vector on $[0,1)^d$ independent of $\mathbf{p}$.
    By a \textit{perturbed lattice}, we understand a stationary random counting measure $\Xi$ on $\R^d$ defined by
    \textcolor{black}{$$\Xi : =  \sum_{i \in \Z^d}\delta_{i+p_i+U}.$$}
Specifically, if $\p$ consists of independent random variables, we speak about \textit{independently perturbed lattice}. If for all $i\in \Z^d$, $p_i=0$ almost surely, then $\Xi$ is called \textit{stationarized lattice}. For an introduction to point processes and random measures, we refer the reader to \cite{DVJ08}.
\end{defn}

\begin{defn}
    By \emph{void probabilities}, we understand a system $\{\mathcal{P}(B), B \in \mathcal{B}^d\}$ of probabilities that there is no point of $\Xi$ lying in the set $B$, i.e.
    \[\mathcal{P}(B):= \P(\Xi (B) = 0).\]
\end{defn}

Multiplicities do not affect void probabilities. Moreover, no moment assumption on the perturbations is needed for local finiteness (see Lemma 2.5 of \cite{DFHL24}).

We estimate the void probabilities of $\Xi$  in terms of the $\alpha-$mixing coefficient of the perturbation field $\mathbf{p}$. Recall that for two sub-$\sigma$-fields $\mathcal{A}_1, \mathcal{A}_2 \subset \mathcal{A}$, we denote
\begin{equation}\label{Eq_mixingOfAlg}
\alpha(\mathcal{A}_1, \mathcal{A}_2) = \textrm{sup}\{|\P(A \cap B)-\P(A)\P(B)|;\, A \in \mathcal{A}_1, B \in \mathcal{A}_2\}.
\end{equation} Specifically, for $C \subset \Z^d$, we denote by $\sigma(p_C)$ the $\sigma-$field generated by $p_C:= \{p_i, i \in C\}$. 
\begin{defn}[$\alpha$-mixing coefficients of a random field]
    For $k,\ell\in \mathbb N\cup\{\infty\}$ and $r>0$, define the $\alpha$-mixing coefficient $\alpha_k,l(r)$ by
\begin{equation}\label{Eq_mixingOfField}
\alpha_{k,\ell}(r)
=
\sup\Bigl\{
\alpha\bigl(\sigma( p_A),\sigma( p_B)\bigr);\,
A,B\subset\Z^d,\,
\#(A)\le k,\,
\#(B)\le \ell,\,
d_{\Z^d}(A,B)\ge r
\Bigr\}, 
\end{equation}
where
$$d_{\Z^d}(A,B)=\inf\{\|i-j\|;\, i \in A, j \in B\}, \qquad  A, B \subset \Z^d.$$
\end{defn}
Although the coefficients $\alpha_{k,\ell}$ are defined in full
generality, the results below require only $\alpha_{1,\infty}$. This
reflects the structure of the argument: after selecting spatially
separated lattice sites, we successively compare the perturbation at one
site with the $\sigma$-field generated by the sites considered previously.

\section{Void probabilities of perturbed lattices}

For a bounded Borel set $B\subset\R^d$, put
\begin{equation}\label{Eq_probabilities}
q_i(u,B)
:=
\P(i+p_i+u\in B),
\qquad
i\in\Z^d,\quad u\in[0,1)^d.
\end{equation}
Thus $q_i(u,B)$ is the probability that the point originating from the
lattice site $i$ falls into $B$ when the stationarizing shift is fixed at
$u$.

Our first result gives the main probabilistic mechanism used throughout
the paper. It separates the problem into two ingredients: the existence
of sufficiently many lattice sites having a non-negligible probability
of hitting the observation set, and the decay of dependence between
spatially separated perturbations.

\begin{theo}[A general void-probability bound]\label{T_general}
Let $B\subset\R^d$ be a bounded Borel set. Suppose that there exist
$\varepsilon>0$ and $M\in\N$ such that
\begin{equation}\label{EQ_good_sites}
\inf_{u\in[0,1)^d}
\#\left(\{i\in\Z^d:q_i(u,B)\geq\varepsilon\}
\right)
\geq M.
\end{equation}
Then there exists a constant $c_d>0$, depending only on the dimension,
such that, for every $r\geq1$,
\[
\mathcal P(B)
\leq
\exp\left\{
-c_d\varepsilon\frac{M}{r^d}
\right\}
+
\frac{M}{r^d}\alpha_{1,\infty}(r).
\]
\end{theo}

Here the parameter $M$ measures how many lattice sites have, uniformly with
respect to the random shift $U$, a non-negligible chance of producing a
point in $B$, while $\varepsilon$ quantifies the size of this probability.
The parameter $r$ plays a different role: it determines the spatial
separation imposed on the selected sites. Increasing $r$ improves the
mixing between the corresponding hitting events, but reduces the number
of sites that can be retained by a factor of order $r^d$. The bound in
Theorem~\ref{T_general} expresses precisely this trade-off.

In applications, the main geometric task is therefore to find
$\varepsilon>0$ for which the number $M$ in
\eqref{EQ_good_sites} admits a useful lower bound. For observation sets
adapted to the underlying lattice, such as unions of lattice cubes, this
can be done without imposing any regularity assumptions on the marginal
distribution of the perturbations.

\begin{prop}[Unions of lattice cubes]\label{C_cubes}
Let
\[
B=
\bigcup_{k\in I}\bigl(k+[0,1)^d\bigr),
\]
where $I\subset\Z^d$ is finite and non-empty.
Then there exists a constant $c>0$, depending only on the dimension and
the marginal law of $p_0$, such that, for every $r\geq1$,
\[
\mathcal P(B)
\leq
\exp\left\{
-c\frac{|B|}{r^d}
\right\}
+
\frac{|B|}{r^d}\alpha_{1,\infty}(r).
\]
\end{prop}

It is worth noting that Proposition~\ref{C_cubes} requires neither moment
conditions nor a density, bounded support, or any other regularity of the
marginal law of $p_0$. The constant $c$ depends on this law, but the form
of the estimate is universal.

The lattice-cube estimate can in turn be used to treat general growing
observation sets. The assumption of non-empty interior below is purely
geometric: it ensures that a sufficiently large dilation contains order
$t^d$ complete unit lattice cubes. 

\begin{theo}[Growing observation sets with non-empty interior]\label{C_convex}
Let $K\subset\R^d$ be a bounded Borel set with non-empty interior.
Then there exist constants $t_0<\infty$ and $c_K>0$, where $t_0$
depends only on $K$ and $c_K$ depends moreover on $d$ and the marginal
law of $p_0$, such that, for every $t\geq t_0$ and $r\geq1$,
\[
\mathcal P(tK)
\leq
\exp\left\{
-c_K\frac{t^d}{r^d}
\right\}
+
|K|\frac{t^d}{r^d}\alpha_{1,\infty}(r).
\]
\end{theo}

Theorem~\ref{C_convex} reduces the asymptotic behaviour of the void
probability to the decay of the mixing coefficient. Optimizing the
separation parameter $r$ gives the following consequences.

\begin{coro}[Mixing rates for growing observation sets]\label{C_growing_rates}
Let $K\subset\R^d$ be a bounded Borel set with non-empty interior.
The constants below may depend on $d$, $K$, the marginal law of $p_0$,
and the constants appearing in the corresponding mixing assumption.

\begin{enumerate}
\item
Assume that there exist constants $C_\alpha,a>0$ such that
\[
\alpha_{1,\infty}(r)\leq C_\alpha e^{-ar},
\qquad r\geq1.
\]
Then there exist constants $c_1,c_2>0$ and $t_0<\infty$ such that
\[
\mathcal P(tK)
\leq
c_1\exp\left\{
-c_2t^{d/(d+1)}
\right\},
\qquad t\geq t_0.
\]

\item
Assume that there exist constants $C_\alpha>0$ and $p>0$ such that
\[
\alpha_{1,\infty}(r)\leq C_\alpha r^{-p},
\qquad r\geq1.
\]
Then
\[
\mathcal P(tK)
=
O\left(
t^{-p}(\log t)^{1+p/d}
\right),
\qquad t\to\infty.
\]
In particular, for every $\eta<p$,
\[
\mathcal P(tK)=O(t^{-\eta}).
\]
\end{enumerate}
\end{coro}

The proofs of Theorem \ref{T_general}, Proposition \ref{C_cubes}, Theorem \ref{C_convex} and Corollary \ref{C_growing_rates} are postponed to Section \ref{S_Proofs}.

\begin{rema}
    Finite-range dependence provides the simplest limiting case of the
preceding bounds. If the perturbation field is $m$-dependent, then
$\alpha_{1,\infty}(r)=0$ for every $r>m$. Consequently, for every bounded
Borel set $K$ with non-empty interior there exists $c_K>0$ such that
\[
\mathcal P(tK)\leq \exp\{-c_Kt^d\}
\]
for all sufficiently large $t$. Similarly, Proposition~\ref{C_cubes}
gives
\[
\mathcal P(B)\leq \exp\{-c|B|\}
\]
for finite unions $B$ of lattice cubes.
\end{rema}

\subsection{Independent perturbations}
The preceding results are designed for dependent perturbation fields.
Their basic mechanism is to select many spatially separated sites with a
non-negligible hitting probability and to control the residual dependence
between them by mixing. Under independence, substantially more information
is available: conditionally on the stationarizing shift, the hitting
events factorize, so that the void probability can be controlled directly
through the total hitting mass
\[
\sum_{i\in\Z^d}q_i(u,B).
\]

This also clarifies the relation with the recent work \cite{ESY26}, where the authors study an unstationarized independently perturbed lattice, denoted by $\Pi$. Under a regularity assumption on
the perturbation tail
\(
p(r):=\P(\|p_0\|\geq r),
\)
they obtain, for large balls, the sharper estimates for $h(r):= \P(\Pi(B(0,r))=0)$:
\[
p(r)^{Cr^d}
\leq
h(r)
\leq
p(r)^{cr^d}.
\]
For Gaussian perturbations this yields
\[
 h(r)
=
\exp\{-\Theta(r^{d+2})\}
\]
at the logarithmic scale. Our Proposition~\ref{P_independent}, in
contrast, holds for an arbitrary i.i.d.\ perturbation distribution and
an arbitrary bounded Borel observation set, but gives its bound in the
implicit form of the total hitting mass. Thus the two results emphasize
different aspects: generality of the perturbation law and observation
set on the one hand, and sharp distribution-dependent asymptotics on the
other.

\begin{prop}\label{P_independent}
Let $\Xi$ be an independently perturbed lattice. 
Then for every bounded Borel set $B\subset \R^d$
\[
\mathcal P(B)
=
\int_{[0,1)^d}
\prod_{i\in\Z^d}
\bigl(1-q_i(u,B)\bigr)\,du,
\]
where the infinite product is understood as the decreasing limit over
finite subsets of $\Z^d$ and $q_i(u,B)$ is defined in \eqref{Eq_probabilities}. In particular,
\[
\mathcal P(B)
\leq
\int_{[0,1)^d}
\exp\left\{
-\sum_{i\in\Z^d}q_i(u,B)
\right\}\,du.
\]
\end{prop}

\begin{proof}
Conditionally on $U=u$, the events
\[
\{i+p_i+u\in B\},
\qquad i\in\Z^d,
\]
are independent. Let $I_n\uparrow\Z^d$ be an increasing sequence of finite
sets. Then
\[
\P(\Xi(B)=0\mid U=u)
=
\lim_{n\to\infty}
\prod_{i\in I_n}
(1-q_i(u,B)).
\]
Using $1-x\leq e^{-x}$ gives
\[
\P(\Xi(B)=0\mid U=u)
\leq
\exp\left\{
-\sum_{i\in\Z^d}q_i(u,B)
\right\}.
\]
Integration over $u\in[0,1)^d$ proves the claim.
\end{proof}

\begin{coro}[Perturbations with a density]\label{C_ind_density}
Assume that $p_0$ has density $f$ and define its periodization by
\[
g(x):=
\sum_{i\in\Z^d}f(x-i),
\qquad x\in\R^d.
\]
Suppose that
\[
c:=
\operatorname*{ess\,inf}_{x\in[0,1)^d}g(x)>0.
\]
Then, for every bounded Borel set $B\subset\R^d$,
\[
\mathcal P(B)\leq e^{-c|B|}.
\]
\end{coro}

\begin{proof}
By Tonelli's theorem,
\[
\begin{aligned}
\sum_{i\in\Z^d}q_i(u,B)
&=
\sum_{i\in\Z^d}
\int_B f(y-u-i)\,\D y=
\int_B
\sum_{i\in\Z^d}f(y-u-i)\,\D y=
\int_B g(y-u)\,\D y.
\end{aligned}
\]
Since $g$ is $\Z^d$-periodic and $g\geq c$ almost everywhere,
\[
\sum_{i\in\Z^d}q_i(u,B)\geq c|B|.
\]
The claim then follows directly from Proposition~\ref{P_independent}.
\end{proof}

\begin{rema}
A sufficient condition for Corollary~\ref{C_ind_density} is the existence
of $c>0$ such that
\[
f(x)\geq c,
\qquad x\in[0,1)^d.
\]
In particular, the condition holds for any continuous density that is
strictly positive on $[0,1]^d$, including non-degenerate Gaussian
perturbations.
\end{rema}

Corollary~\ref{C_ind_density} illustrates the additional flexibility provided by independence. Once the periodized density is bounded away from zero, an exponential-in-volume bound holds for arbitrary bounded Borel observation sets, without any geometric regularity assumption on \(B\). For dependent perturbations, the product structure is no longer available, and the role of geometry is instead to provide sufficiently many separated sites to which Theorem~\ref{T_general} can be applied.

\section{Proofs}\label{S_Proofs}
We first prove the general bound of Theorem~\ref{T_general}. We then
verify its good-sites condition for unions of lattice cubes, which yields
Proposition~\ref{C_cubes}. Finally, we approximate a growing observation
set from inside by such cubes to obtain Theorem~\ref{C_convex}.

\begin{proo}[of Theorem \ref{T_general}]
Fix $u\in[0,1)^d$ and choose $M$ indices $i$ for which
$q_i(u,B)\geq\varepsilon$. Put $L:=\lceil r\rceil$ and partition $\Z^d$ into the $L^d$ residue classes modulo $L$. By the pigeonhole principle, one of these classes contains at least
$\frac{M}{L^d}
$ of the selected indices. Since any two indices in the same residue class are separated by distance at least $L\geq r$, we can choose an $r$-separated subset of these indices of cardinality

$$
N:=
\max\left\{
1,
\left\lfloor\frac{M}{L^d}\right\rfloor
\right\}.
$$
Set $c_d:= 2^{-(d+1)}$. Since $r\leq L\leq2r$, 
$$
N\geq \frac{M}{2L^d} \geq c_d\frac{M}{r^d},
$$
while
$$
N-1
\leq
\frac{M}{L^d}
\leq
\frac{M}{r^d}.
$$
Denote the selected indices by $i_1,\ldots,i_N$ and put
$$
E_k:=\{i_k+p_{i_k}+u\notin B\}, \quad k=1,\ldots, N.
$$
By this construction, 
$$\{\Xi(B)=0\} \subseteq \bigcap_{k=1}^N E_k.$$
Since the indices $i_1,\ldots,i_{N}$ are mutually separated by distance at least $r$, the $\sigma$-fields
$$
\sigma(p_{i_k})
\qquad\text{and}\qquad
\sigma(p_{i_1},\ldots,p_{i_{k-1}})
$$
are separated by distance at least $r$ for every $k\geq2$. Therefore,
$$
\P\left(
E_k\cap\bigcap_{\ell=1}^{k-1}E_\ell
\right)
\leq
\P(E_k)
\P\left(
\bigcap_{\ell=1}^{k-1}E_\ell
\right)
+
\alpha_{1,k-1}(r).
$$
Iterating this estimate gives
$$
\P\left(
\bigcap_{\ell=1}^{N}E_\ell
\right)
\leq
\prod_{\ell=1}^{N}\P(E_\ell)
+
\sum_{k=2}^{N}\alpha_{1,k-1}(r) \leq
\prod_{\ell=1}^{N}\P(E_\ell)
+
(N-1)\alpha_{1,\infty}(r).
$$

Since $q_{i_k}(u,B)\geq\varepsilon$,
$$
\P(E_k)=1-q_{i_k}(u,B)\leq e^{-\varepsilon},
$$
and therefore
$$
\P(\Xi(B)=0\mid U=u)
\leq
e^{-\varepsilon N}
+
(N-1)\alpha_{1,\infty}(r).
$$
Using the bounds on $N$ gives
$$
\P(\Xi(B)=0\mid U=u)
\leq
\exp\left\{
-c_d\varepsilon\frac{M}{r^d}
\right\}
+ \frac{M}{r^d}\alpha_{1,\infty}(r).
$$

The right-hand side is uniform in $u$, so integration over
$u\in[0,1)^d$ proves the claim.

\end{proo}

\begin{proo}[of Proposition \ref{C_cubes}]
We show that there exists $\varepsilon_0>0$,
depending only on the marginal law of $p_0$, such that condition
\eqref{EQ_good_sites} holds with $M= |I|=|B|.$

Since the law of $p_0$ is tight, there exists $R<\infty$ such that
$$
\P(p_0\in[-R,R]^d)\geq\frac12.
$$
Fix $u\in[0,1)^d$. The shifted cubes
$$
[0,1)^d-u-j,
\qquad j\in\Z^d,
$$
form a partition of $\R^d$, and at most $(2R+2)^d$ of them intersect
$[-R,R]^d$. Hence there exists $j(u)\in\Z^d$ such that
$$
\P\bigl(p_0\in[0,1)^d-u-j(u)\bigr)
\geq
\varepsilon_0,
$$
with
$$
\varepsilon_0
:=
\frac{1}{2(2R+2)^d}.
$$

By stationarity of $\mathbf{p}$, 
$p_{k+j(u)}$ has the same distribution as $p_0$ for every $k\in I$. Therefore
$$
\begin{aligned}
q_{k+j(u)}(u,k+[0,1)^d)
&=
\P(k+j(u)+p_{k+j(u)}+u\in k+[0,1)^d)
\\
&=
\P(j(u)+p_0+u\in[0,1)^d)
\\
&\geq
\varepsilon_0.
\end{aligned}
$$
Since $k+[0,1)^d\subset B$, also $q_{k+j(u)}(u,B)\geq\varepsilon_0.$ Moreover, the indices $k+j(u)$ are distinct as $k$ ranges over $I$.
Thus, for every $u\in[0,1)^d$,
$$
\#\big(
\{i\in\Z^d:q_i(u,B)\geq\varepsilon_0\}
\big)
\geq |I|.$$

Applying Theorem~\ref{T_general} yields
$$
\mathcal P(B)
\leq
\exp\left\{
-c\frac{|B|}{r^d}
\right\}
+
\frac{|B|}{r^d}\alpha_{1,\infty}(r),
\qquad r\geq1,
$$
for a suitable constant $c >0$.
\end{proo}

\begin{proo}[of Theorem \ref{C_convex}]
By stationarity of $\Xi$, we may translate $K$ and assume without loss of
generality that $0\in\operatorname{int}(K)$. Hence there exists
$\rho_K>0$ such that
\[
[-\rho_K,\rho_K]^d\subset K,
\]
and consequently 
$[-\rho_Kt,\rho_Kt]^d\subset tK.$
Put $t_0:=\rho_K^{-1}.$ For $t\geq t_0$, let
\[
I_t
:=
\left\{
k\in\Z^d:
k+[0,1)^d\subset tK
\right\},
\]
and define
\[
B_t
:=
\bigcup_{k\in I_t}
\bigl(k+[0,1)^d\bigr).
\]
Since $\rho_Kt\geq1$, the set $B_t$ is non-empty and contains all unit
cubes
\[
Q_k:=k+[0,1)^d,
\qquad
k\in
\{0,\ldots,\lfloor\rho_Kt\rfloor-1\}^d.
\]
Therefore
\begin{equation}\label{eq_lowerB}
    |B_t|
\geq
\lfloor\rho_Kt\rfloor^d
\geq
\left(\frac{\rho_K}{2}\right)^dt^d.
\end{equation}
On the other hand, since $B_t\subset tK$,
\begin{equation}\label{eq_upperB} 
|B_t|
\leq
|tK|
=
|K|t^d.
\end{equation}

Applying Proposition~\ref{C_cubes} to $B_t$ gives
\[
\mathcal P(B_t)
\leq
\exp\left\{
-c\frac{|B_t|}{r^d}
\right\}
+
\frac{|B_t|}{r^d}\alpha_{1,\infty}(r),
\]
where $c>0$ depends only on $d$ and the marginal law of $p_0$. Consequently, by monotonicity of void probabilities and the bounds \eqref{eq_lowerB} and \eqref{eq_upperB}
\[
\mathcal P(tK)\leq \mathcal{P} (B_t)
\leq
\exp\left\{
-c_K\frac{t^d}{r^d}
\right\}
+
|K|\frac{t^d}{r^d}\alpha_{1,\infty}(r),
\]
with
\[
c_K
:=
c\left(\frac{\rho_K}{2}\right)^d.
\]
\end{proo}

\begin{proo}[of Corollary \ref{C_growing_rates}]
Assume first that
\(\alpha_{1,\infty}(r)\leq C_\alpha e^{-ar}\) and choose
\(
r=t^{d/(d+1)}.
\)
Then by Theorem \ref{C_convex}
\[
\mathcal P(tK)
\leq
\exp\left\{
-c_Kt^{d/(d+1)}
\right\}
+
C_\alpha|K|t^{d/(d+1)}
\exp\left\{
-at^{d/(d+1)}
\right\}
\]
for all $t\geq t_0$. Absorbing the polynomial factor into the exponential proves the claim.

Suppose now that
\(
\alpha_{1,\infty}(r)\leq C_\alpha r^{-p}.
\)
Theorem~\ref{C_convex} gives
\begin{equation}\label{EQ_poly_1}
\mathcal P(tK)
\leq
\exp\left\{
-c_K\frac{t^d}{r^d}
\right\}
+
C_\alpha|K|\frac{t^d}{r^{p+d}}.
\end{equation}
Set
\(
x:=\frac{t^d}{r^d}.
\)
Since
\(
r^{-p}=t^{-p}x^{p/d},
\)
inequality~\eqref{EQ_poly_1} becomes
\[
\mathcal P(tK)
\leq
e^{-c_Kx}
+
C_\alpha|K|t^{-p}x^{1+p/d}.
\]

Now choose
\[
r=
\frac{t}{(A\log t)^{1/d}},
\]
where $A>p/c_K$. For all sufficiently large $t$, this choice satisfies
$r\geq1$, and
\(
x=A\log t.
\)
Consequently,
\[
e^{-c_Kx}
=
t^{-c_KA}
=
O(t^{-p}),
\]
whereas
\[
t^{-p}x^{1+p/d}
=
A^{1+p/d}
t^{-p}(\log t)^{1+p/d}.
\]
It follows that
\[
\mathcal P(tK)
=
O\left(
t^{-p}(\log t)^{1+p/d}
\right).
\]

Finally, for every $\eta<p$,
\(
(\log t)^{1+p/d}
=
O(t^{p-\eta}),
\)
and hence
\[
\mathcal P(tK)=O(t^{-\eta}).
\]

\end{proo}

\section*{Acknowledgement}
This work was supported by the Czech Science Foundation, project no. 26-22373N.

\end{document}